\documentclass[11pt]{amsart}
\usepackage{amsaddr}
\usepackage{amsmath, amssymb, amsthm, bm, booktabs}
\usepackage{geometry}
\usepackage{enumitem}
\usepackage{tikz}
\usetikzlibrary{shapes.geometric, positioning}
\usepackage{setspace}
\usepackage{array}
\usepackage{hyperref}
\usepackage{parskip}
\newtheorem{theorem}{Theorem}
\newtheorem{proposition}{Proposition}
\newtheorem{lemma}{Lemma}
\newtheorem{definition}{Definition}
\newtheorem{corollary}{Corollary}
\newtheorem{remark}{Remark}
\title[Complexity of Recognizing SDP Exactness for Max-Cut]{The Complexity of Recognizing SDP Exactness for the Maximum Cut Problem}
\author{Avinash Bhardwaj}
\email{abhardwaj@iitb.ac.in}
\address{Department of Industrial Engineering and Operations Research,\\ Indian Institute of Technology Bombay,\\ Mumbai, India 400076}
\date{September 2026}

\begin{document}
\begin{abstract}
The standard semidefinite programming (SDP) relaxation of Max-Cut is exact when its optimum equals the maximum cut value. Delorme and Poljak resolved NP-completeness of recognizing exactness for weighted graphs and left the unweighted case open. We show that recognition is NP-complete even for connected simple unweighted graphs, and hence strongly NP-complete for nonnegative integer edge weights.

The reduction provides an explicit SDP optimum and makes the additive integrality gap equal to the minimum number of unsatisfied clauses in the source formula. Recognition remains NP-complete even when an exact rational optimal primal--dual pair is supplied.

We also establish strong NP-hardness of recognizing exactness of the Frieze--Jerrum Max-$k$-Cut relaxation for every fixed $k\ge3$, even for connected graphs with nonnegative integer edge weights. An independent bounded-weight construction gives a second proof for Max-Cut. Finally, reductions preserving the additive gap up to explicit factors establish strong NP-completeness of exactness recognition for a basic Max-DiCut SDP and NP-hardness for a Max-Bisection SDP.
\end{abstract}
\maketitle
\section{Introduction}
\label{sec:introduction}
The Maximum Cut (Max-Cut) problem is a fundamental combinatorial optimization problem that seeks to partition the vertices of an undirected, edge-weighted graph into two disjoint sets such that the total weight of the edges crossing the partition is maximized. Let $G=(V,E)$ be an undirected graph on $n$ vertices with nonnegative edge-weight matrix $W\in\mathbb R^{n\times n}$. The matrix $W$ is symmetric with zero diagonal, and $W_{ij}=0$ whenever $\{i,j\}\notin E$. Computational statements throughout the manuscript use integer or rational weights represented in binary. The integer formulation of
Max-Cut is~\cite{GW95}:
\begin{equation}
    \operatorname{MC}(G) = \max_{x \in \{-1, 1\}^n} \frac{1}{4} \sum_{i \neq j} W_{ij} (1 - x_i x_j).\label{eq:maxcut-intro}
\end{equation}
Since Max-Cut is strongly NP-hard~\cite{karp1972reducibility}, extensive research has focused on its continuous relaxations. The foundational work of Goemans and Williamson~\cite{GW95} introduced a Semidefinite Programming (SDP) relaxation that lifts the scalar variables $x_i$ to unit vectors $v_i \in \mathbb{R}^n$:
\begin{equation}
    \operatorname{SDP}(G) = \max_{X \succeq 0, X_{ii}=1} \frac{1}{4} \sum_{i \neq j} W_{ij} (1 - X_{ij}),\label{eq:sdp-intro}
\end{equation}
where $X$ is the Gram matrix of the unit vectors, with entries $X_{ij}=\langle v_i,v_j\rangle$. Combined with randomized hyperplane rounding, this relaxation yields an approximation algorithm with ratio $\alpha_{GW}\approx 0.878$.

The relaxation is \emph{exact} if $\operatorname{SDP}(G)=\operatorname{MC}(G)$, equivalently, if it has a rank-one optimal solution. This requires the existence of such a solution; other optimal solutions may have higher rank. In this manuscript, we study the complexity of recognizing exactness from the input graph $G$.

\subsection{Previous work}
The exactness of continuous relaxations for Max-Cut is deeply rooted in the spectral bounds established by Delorme and Poljak in their seminal series of papers~\cite{DP93_II,DP93_I}. Specifically, given a graph $G$, Delorme and Poljak~\cite{DP93_II,DP93_I} studied the eigenvalue bound
\[
\phi(G)=\min_{u\in\mathbb R^n,\ \sum_i u_i=0}\frac n4\lambda_{\max}\bigl(L+\operatorname{Diag}(u)\bigr),
\]
where $L$ is the weighted Laplacian of $G$. This bound equals the optimum of~\eqref{eq:sdp-intro},as explained through semidefinite duality in~\cite{PR95}. Delorme and Poljak proved NP-completeness of recognizing exactness for weighted graphs~\cite[Corollary~3.3]{DP93_II}. However, their reduction from Exact Sum does not establish hardness for polynomially bounded weights. They explicitly left open the recognition problem for simple unweighted graphs~\cite[p.~324]{DP93_II}.

Laurent and Poljak~\cite{laurent1996positive} give an algebraic description that helps explain the connection between exactness and satisfiability. They study minimization of objectives of the form $\langle bb^\top,X\rangle$ over the elliptope, the feasible region of \eqref{eq:sdp-intro}. On a cut matrix $X=xx^\top$, this objective equals $(b^\top x)^2$. When the SDP minimum is zero, exactness therefore amounts to finding signs $x_i\in\{-1,1\}$ such that $\sum_i b_i x_i=0$. This connection underlies their hardness result for recognizing attainment by a cut matrix. 

Our constructions use a sum of such squares, one for each clause. Auxiliary vectors make every squared term vanish in the SDP, whereas a sign assignment makes them all vanish if and only if the source formula is satisfiable. In the unweighted construction, these clause constraints are
implemented by edge-disjoint copies of $K_6$.

A complementary line of exploration identifies graph families and operations for which exactness holds. Hong, Lee, and Wei~\cite{hong2021} study ranks of optimal solutions and give a signed-Laplacian certificate. Mirka and Williamson~\cite{MW2X} investigate uniqueness and semidefinite rank. Bhardwaj, Gogoi, Naryanan and Pathapati~\cite{GBNP25} establish further exact families and exactness-preserving operations, and give examples distinguishing uniqueness of the maximum cut from uniqueness of the SDP optimum. 

These structural results also provide context for the recognition problem. As recalled in Section~\ref{sec:np_completeness}, a given cut determines a slack matrix whose positive semidefiniteness certifies that the cut is SDP-optimal. Consequently, on any graph class where a maximum cut can be computed in polynomial time, exactness can also be decided in polynomial time. This applies, for example, to planar graphs~\cite{OD73} and graphs of fixed bounded treewidth~\cite{arnborg1989linear}.

\subsection{Contributions}
Our main result answers the unweighted recognition question posed by Delorme and Poljak.

\begin{theorem}\label{thm:main}
Deciding whether the standard Max-Cut SDP relaxation is exact is NP-complete even for connected simple unweighted graphs. Consequently, the problem is strongly NP-complete for graphs with nonnegative integer edge weights.
\end{theorem}

The proof constructs a graph whose SDP optimum is known explicitly, but whose maximum cut attains that optimum precisely when a given Boolean formula is satisfiable. The source problem is monotone NAE-4-SAT: each clause contains four unnegated variables and is satisfied when their values are not all equal. We use linear instances, in which every clause contains four distinct variables and any two distinct clauses share at most one variable. The NP-completeness of this restriction follows from Schaefer's theorem together with the corollary of Kun's deterministic Sparse Incomparability Theorem stated in Lemma~\ref{lem:kun}.

Each clause is represented by a $K_6$ on its four variable vertices and two private auxiliary vertices. Linearity makes these cliques edge-disjoint, so their contributions to both the cut and SDP objectives can be added. An explicit vector assignment makes every clique contribute $9$ to the SDP objective, attaining its individual upper bound. 

For a fixed Boolean assignment to the variable vertices, optimizing the auxiliary vertices gives a cut contribution of $9$ when the clause is satisfied and $8$ otherwise. Consequently, if the source formula $\Phi$ has $m$ clauses, then
\[
\operatorname{SDP}(G)=9m, \qquad \operatorname{MC}(G) = 8m+\operatorname{OPT}_{\mathrm{NAE}}(\Phi),
\]
where $\operatorname{OPT}_{\mathrm{NAE}}(\Phi)$ is the maximum number of simultaneously satisfied clauses. Thus the additive integrality gap equals the minimum number of unsatisfied clauses, and the relaxation is exact if and only if $\Phi$ is satisfiable. A further vertex-identification step ensures connectedness while preserving both optimum values.

The construction also supplies an exact rational optimal primal--dual pair of polynomial encoding length. Corollary~\ref{cor:supplied-optimum} shows that recognition remains NP-complete for simple unweighted graphs when this pair is included in the input. Thus, even when the continuous optimum is explicitly known and certified, deciding whether a cut attains it remains NP-complete.

Our second main result establishes recognition hardness for the Frieze--Jerrum SDP relaxation of Max-$k$-Cut.

\begin{theorem}\label{thm:kcut-main}
For every fixed integer $k\ge3$, recognizing exactness of the Frieze--Jerrum Max-$k$-Cut SDP relaxation is strongly NP-hard, even for connected graphs with nonnegative integer edge weights.
\end{theorem}

A Max-$k$-Cut partition can be viewed as an assignment of one of $k$ colors to each vertex, with an edge contributing its weight when its endpoints receive different colors. To encode a Boolean assignment within such a partition, we introduce $k-2$ anchor vertices shared by all clauses. The construction ensures that, in any partition attaining the SDP bound, the anchors occupy distinct colors and no clause vertex uses an anchor color. The two remaining colors therefore encode the Boolean assignment. An explicit vector assignment attains the SDP bound independently of satisfiability. The proof also expresses the additive gap as the minimum of a sum of squared deviations from the balance required within each clause.

We also give an independent proof of strong NP-completeness of Max-Cut SDP exactness recognition for nonnegative integer weights, using Monotone $\{3,6\}$-in-9 SAT. This construction avoids Kun's theorem and permits contributions from different clauses to accumulate as edge weights.

Two further reductions transfer Max-Cut exactness recognition to related partitioning problems. For Max-DiCut, the objective counts the weights of arcs directed from the first part to the second. Recognizing exactness of the basic SDP~\eqref{eq:dicut-sdp} is strongly NP-complete even for strongly connected digraphs in which every adjacent pair has positive unequal opposite arc weights and every vertex has equal weighted indegree and outdegree. For Max-Bisection, the two parts must have equal size. Recognizing exactness of SDP~\eqref{eq:bisection-sdp} is NP-hard even on connected simple unweighted graphs. Both reductions preserve the Max-Cut additive integrality gap up to an explicit factor. These claims concern the particular SDP formulations displayed in the corresponding subsections.

The rest of the manuscript is organized as follows: Section~\ref{sec:np_completeness} recalls optimality certificates and their signed-Laplacian interpretation. In Section~\ref{sec:unweighted}, we prove Theorem~\ref{thm:main} and the supplied-certificate corollary. Section~\ref{sec:weighted} gives the independent weighted reduction. Theorem~\ref{thm:kcut-main} and the other extensions are resolved in Section~\ref{sec:exactness-applications}. Section~\ref{sec:conclusion} concludes with open questions about stronger relaxations.

\section{Optimality Certificates}
\label{sec:np_completeness}
We recall the certificate underlying NP membership~\cite{DP93_II} in a form that also applies to the anchored Max-DiCut SDP. Let $\mathbb S^N$ denote the real symmetric $N\times N$ matrices, with inner product $\langle A,B\rangle=\operatorname{Tr}(AB)$.

\begin{lemma}[Rank-one optimality certificate]\label{lem:rank-one-certificate}
Let $B\in\mathbb S^N$ be rational and consider
\[
\max\{\langle B,X\rangle:X\succeq0,\ X_{ii}=1\}.
\]
For $z\in\{-1,1\}^N$, define $y_i=z_i(Bz)_i$. Then $zz^\top$ is optimal if and only if $\operatorname{Diag}(y)-B\succeq0$. The vector $y$ has polynomial encoding length and this condition can be checked in polynomial time.
\end{lemma}
\begin{proof}
The primal is strictly feasible at $I$ and has compact feasible region. Its dual minimizes $\sum_i y_i$ subject to $S=\operatorname{Diag}(y)-B\succeq0$, and strong duality and dual attainment hold. If $zz^\top$ is optimal, complementary slackness gives $Sz=0$, forcing $y_i=z_i(Bz)_i$. Conversely, for this choice of $y$, dual feasibility and $\sum_i y_i=z^\top Bz$ certify optimality. Rational matrix--vector multiplication gives $y$ with polynomial encoding length. Positive semidefiniteness of a rational symmetric matrix is decidable in polynomial bit complexity by exact
Gaussian elimination; see~\cite[Section~2.3.3]{LaurentNotes}.
\end{proof}

\begin{corollary}\label{cor:np-membership}
Recognizing exactness of the standard Max-Cut SDP for rational edge weights belongs to NP.
\end{corollary}
\begin{proof}
The objective is $\langle L,X\rangle/4$, where $L=D-W$. By Lemma~\ref{lem:rank-one-certificate}, a certificate consists only of a cut vector $x\in\{-1,1\}^n$. The verifier forms
\begin{equation}\label{eq:cut-certificate}
y_i=x_i(Lx)_i=\sum_{j\ne i}W_{ij}(1-x_ix_j),\qquad
S=\operatorname{Diag}(y)-L,
\end{equation}
and checks $S\succeq0$. If the test succeeds, the cut and dual values both equal $\frac14\sum_i y_i$. Conversely, any optimal cut of an exact instance passes the test.
\end{proof}
Thus each $y_i$ is twice the weight of cut edges incident to vertex $i$. In particular, the certificate does not require an arbitrary real solution of the dual SDP. If a maximum cut is already available, its slack matrix decides exactness.

\subsection{Signed-Laplacian interpretation}
\label{sec:structural_characterization}
Hong, Lee, and Wei~\cite[Lemma~1]{hong2021} express rank-one optimality through a positive semidefinite Laplacian after changing edge signs. We recall the equivalent certificate and its interpretation.
\begin{definition}[Cut-signed Laplacian]
For a cut vector $x\in\{-1,1\}^n$, set $W^{(x)}_{ij}=-W_{ij}x_ix_j$. The cut-signed Laplacian is $L^{(x)}=D^{(x)}-W^{(x)}$, where $D^{(x)}_{ii}=\sum_{j\ne i}W^{(x)}_{ij}$. Crossing edges retain their weights; uncut edges have their signs reversed.
\end{definition}
\begin{proposition}\label{thm:structural_exactness}
A graph is exact if and only if there is a cut vector $x$ with $L^{(x)}\succeq0$.
\end{proposition}
\begin{proof}
For the slack matrix in~\eqref{eq:cut-certificate}, $S_{ij}=W_{ij}$ for $i\ne j$ and $S_{ii}=-\sum_{j\ne i}W_{ij}x_ix_j$. With $P=\operatorname{Diag}(x)$, entrywise calculation gives $PSP=L^{(x)}$. Since $P=P^{-1}$, this congruence preserves positive semidefiniteness. The claim follows from Lemma~\ref{lem:rank-one-certificate}.
\end{proof}
For every $z\in\mathbb R^n$,
\[
z^\top L^{(x)}z
=
\sum_{i<j}W^{(x)}_{ij}(z_i-z_j)^2.
\]
When the original weights are nonnegative, crossing edges contribute nonnegative terms and uncut edges contribute nonpositive terms. Let $G_{\mathrm{cut}}$ and $G_{\mathrm{uncut}}$ be the corresponding spanning subgraphs, retaining their original weights. Then
\[
L^{(x)}=L(G_{\mathrm{cut}})-L(G_{\mathrm{uncut}}).
\]
Exactness therefore means that some cut satisfies $L(G_{\mathrm{cut}})\succeq L(G_{\mathrm{uncut}})$: its crossing-edge contribution dominates the uncut-edge contribution for every real vector $z$. This is a spectral condition on a bipartite core relative to the chosen partition, rather than a requirement that the whole graph be bipartite. Graph-join examples in~\cite{GBNP25} illustrate how internal edges can coexist with this domination.

This instance-dependent condition differs from exactness of a polyhedral relaxation for every objective on a fixed graph. For graphs with no $K_5$ minor, cycle inequalities and edge bounds describe the cut polytope~\cite{barahona1986cut}; the same restriction does not ensure SDP exactness. The unweighted triangle, for example, has cut value $2$ and SDP value $9/4$.

\section{NP-Completeness for Connected Simple Unweighted Graphs}
\label{sec:unweighted}
We prove Theorem~\ref{thm:main} by reducing a Boolean satisfiability problem to Max-Cut SDP exactness. Given a formula, we construct a graph whose SDP optimum is known explicitly and whose maximum cut attains that value if and only if the formula is satisfiable. We first establish NP-completeness for the restricted satisfiability problem used in the reduction.

\subsection{The Base Combinatorial Problem}
We use \textnormal{\scshape Monotone Not-All-Equal (NAE)-4-SAT} as the base problem. An instance consists of a set of Boolean variables and a collection of clauses, each containing four distinct, unnegated variables. A clause is satisfied when its variables are not all assigned the same value. An instance is \emph{linear} if any two distinct clauses share at most one variable.

The following corollary of Kun's theorem \cite[Theorem~1]{kun2013constraints} allows us to impose this restriction while preserving satisfiability. We state it explicitly as it is the only consequence of that theorem needed here.

\begin{lemma}[Corollary of {\cite[Theorem~1]{kun2013constraints}}]
\label{lem:kun}
There is a deterministic polynomial-time algorithm that, given a collection $\Phi$ of unnegated NAE constraints with four argument positions each, returns a linear monotone NAE-4-SAT instance $\Phi'$ such that $\Phi'$ is satisfiable if and only if $\Phi$ is. Variables may repeat within an input constraint; every output clause contains four distinct variables.
\end{lemma}

\begin{proof}
Use the Boolean constraint relation
\[
R_{\mathrm{NAE}} = \{0,1\}^4 \setminus \{(0,0,0,0),(1,1,1,1)\}.
\]
For a fixed finite constraint relation and a fixed girth bound, Kun's theorem constructs, in deterministic polynomial time, an instance using the same relation that preserves satisfiability and exceeds the prescribed girth bound. Here we take the bound to be two. In the notation of \cite[Theorem~1]{kun2013constraints}, the parameters are $t=3$ and $k=2$: the Boolean domain has size two, which is less than $t$.

In Kun's convention, a constraint with a repeated variable constitutes a cycle of length one. Two distinct constraints containing at least two common distinct variables give a cycle of length two. The output has girth greater than two, so neither can occur. Every output constraint therefore has four distinct variables, and any two distinct output constraints share at most one variable.

The output uses the same unnegated NAE relation. Since this relation is invariant under coordinate permutations, each output constraint can be read as an unordered four-element clause. The resulting formula is the required linear instance $\Phi'$.
\end{proof}

\begin{lemma}\label{lem:linear-nae4}
\textnormal{\scshape Monotone NAE-4-SAT} is NP-complete even when restricted to linear instances.
\end{lemma}

\begin{proof}
First allow variables to repeat within a constraint. This is the Boolean constraint satisfaction problem for the relation $R_{\mathrm{NAE}}$ above. By Schaefer's Dichotomy Theorem~\cite{schaefer1978complexity}, this problem is NP-complete unless the relation belongs to one of six tractable classes. We exclude these possibilities.

The relation is neither 0-valid nor 1-valid, since it contains neither constant tuple. The assignments $(1,0,0,0)$ and $(0,1,0,0)$ show that it is not closed under conjunction, and their complements show that it is not closed under disjunction. Thus the relation is neither Horn nor dual-Horn. The coordinatewise majority of
\[
(1,0,0,0),\qquad (0,1,0,0),\qquad (0,0,1,0)
\]
is $(0,0,0,0)$, so the relation is not bijunctive. Finally, its cardinality is $14$, which is not a power of two, so it is not affine.

Lemma~\ref{lem:kun} transforms these instances into linear monotone NAE-4-SAT instances in deterministic polynomial time, preserving satisfiability. This proves NP-hardness. Membership in NP follows by checking a proposed Boolean assignment against every clause.
\end{proof}

\subsection{The \texorpdfstring{$K_6$}{K6} Construction}
Given a restricted NAE-4-SAT instance, we construct a simple unweighted graph $G = (V, E)$. We may assume that the source instance has at least one clause: adjoining a clause on four fresh variables preserves satisfiability and the linearity restriction. For each clause $C_j = \{x_{j_1}, x_{j_2}, \allowbreak x_{j_3}, x_{j_4}\}$, we introduce two private auxiliary vertices, $d_{j,1}$ and $d_{j,2}$. We construct a clique ($K_6$) on these 6 vertices.

The global graph $G$ is the union of these $m$ cliques. As the source instance is linear, any two clauses share at most one variable. Consequently, any two distinct $K_6$ blocks share at most one vertex, and therefore no edges. $G$ is simple and unweighted, and has exactly $15m$ edges, partitioned into $m$ edge-disjoint copies of $K_6$.

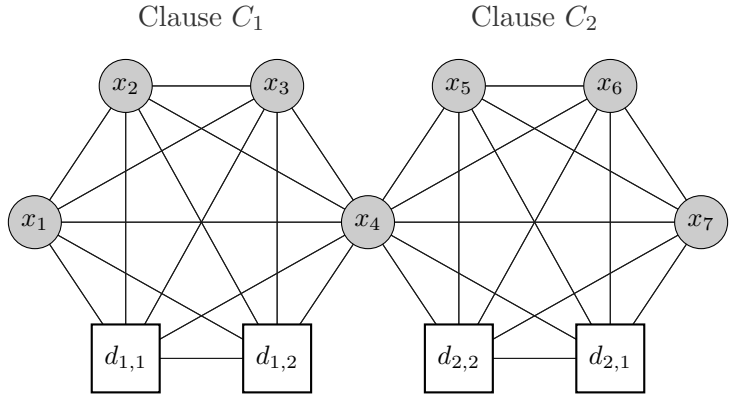
\begin{figure}[htbp]
    \centering
    \begin{tikzpicture}[
        var_node/.style={circle, draw=black, fill=black!20, inner sep=1pt, minimum size=7mm, text=black, font=\sffamily\small},
        aux_node/.style={regular polygon, regular polygon sides=4, draw=black, fill=white, thick, inner sep=1pt, minimum size=7mm, font=\sffamily\small},
        edge/.style={black!90}
    ]
    
    % Coordinates for C1 (Left K6)
    \coordinate (x4) at (0,0);
    \coordinate (x3) at (-1.2, 1.8);
    \coordinate (x2) at (-3.2, 1.8);
    \coordinate (x1) at (-4.4, 0);
    \coordinate (d11) at (-3.2, -1.8);
    \coordinate (d12) at (-1.2, -1.8);
    
    % Coordinates for C2 (Right K6)
    \coordinate (x5) at (1.2, 1.8);
    \coordinate (x6) at (3.2, 1.8);
    \coordinate (x7) at (4.4, 0);
    \coordinate (d21) at (3.2, -1.8);
    \coordinate (d22) at (1.2, -1.8);
    
    % Draw edges for C1
    \foreach \i in {x4, x3, x2, x1, d11, d12} {
        \foreach \j in {x4, x3, x2, x1, d11, d12} {
            \draw[edge] (\i) -- (\j);
        }
    }
    
    % Draw edges for C2
    \foreach \i in {x4, x5, x6, x7, d21, d22} {
        \foreach \j in {x4, x5, x6, x7, d21, d22} {
            \draw[edge] (\i) -- (\j);
        }
    }
    
    % Draw nodes over the edges
    % Shared variable
    \node[var_node] at (x4) {\(x_4\)};
    
    % C1 Variables and Auxiliaries
    \node[var_node] at (x1) {\(x_1\)};
    \node[var_node] at (x2) {\(x_2\)};
    \node[var_node] at (x3) {\(x_3\)};
    \node[aux_node] at (d11) {\(d_{1,1}\)};
    \node[aux_node] at (d12) {\(d_{1,2}\)};
    
    % C2 Variables and Auxiliaries
    \node[var_node] at (x5) {\(x_5\)};
    \node[var_node] at (x6) {\(x_6\)};
    \node[var_node] at (x7) {\(x_7\)};
    \node[aux_node] at (d21) {\(d_{2,1}\)};
    \node[aux_node] at (d22) {\(d_{2,2}\)};
    
    % Labels for the cliques
    \node[text=black!80] at (-2.2, 2.7) {Clause $C_1$};
    \node[text=black!80] at (2.2, 2.7) {Clause $C_2$};
    
    \end{tikzpicture}
    \caption{\small The graph \(G\) constructed from the restricted NAE-4-SAT formula \(\Phi = C_1 \land C_2\), where \(C_1 = \{x_1, x_2, x_3, x_4\}\) and \(C_2 = \{x_4, x_5, x_6, x_7\}\). Because the hypergraph is linear, the clauses share exactly one variable (\(x_4\)) and no edge. Variable vertices are plotted as filled circles, while the private auxiliary vertices (\(d_{j,1}, d_{j,2}\)) are plotted as empty squares. The resulting global graph consists of two edge-disjoint copies of \(K_6\).}
    \label{fig:k6_construction}
\end{figure}
\subsection{Continuous Bounds via Geometric Embedding}
For unit vectors $w_1, \allowbreak w_2 \ldots,w_6$, the contribution of a $K_6$ block to the SDP objective is
\begin{equation}
\begin{aligned}
    F(w_1,\ldots,w_6)
    &:= \frac14 \sum_{i\ne k}
        (1-\langle w_i,w_k\rangle)\\
    &= 9-\frac14
        \left\|\sum_{i=1}^{6}w_i\right\|^2
    \le 9.
\end{aligned}
\end{equation}
Since $\left\|\sum w_i\right\|^2 \ge 0$, the SDP objective value for a single $K_6$ block is bounded above by $9$. As the $m$ cliques in $G$ are edge-disjoint, the global SDP objective value is bounded above by $9m$.

We claim that the SDP relaxation of $G$ always achieves an objective value of exactly $9m$. It is not immediate since a variable vertex may belong to several blocks and must receive the same vector in all of them. To see this, assign the $n$ global variables to mutually orthogonal unit vectors $v_1, \dots, v_n \in \mathbb{R}^n$. For a given clause $C_j$, the sum of its four corresponding orthogonal vectors is $S_j = \sum_{k=1}^4 v_{j_k}$. The squared norm is $\|S_j\|^2 = 4$. 

We assign the value of $-\frac{1}{2} S_j$ to the two local auxiliary vertices($u_{j,1} = u_{j,2} = -\frac{1}{2} S_j$). The squared norm of each of these vectors is $\|u_{j,1}\|^2 = \frac{1}{4}(4) = 1$, proving they are valid unit vectors. The sum of the 6 vectors for this block is $S_j - \frac{1}{2}S_j - \frac{1}{2}S_j = 0$. Thus, every $K_6$ block contributes exactly \(9\) to the SDP objective. The continuous SDP relaxation yields $\operatorname{SDP}(G) = 9m$.

\subsection{Combinatorial Exactness}
For the integer Max-Cut to achieve $9m$, every $K_6$ clique must independently achieve a cut of 9. Let $z \in \{-1, 1\}^{|V|}$ be a Boolean assignment. The integer cut for a single $K_6$ block is $9 - \frac{1}{4}(\sum z_i)^2$. To achieve a cut of 9, the sum of the 6 vertices must be 0, requiring a strict 3-versus-3 partition.

Let $s_j = \sum_{k=1}^4 z_{j_k} \in \{-4, -2, 0, 2, 4\}$ be the sum of the four variables, and $t_j = z_{d_{j,1}} + z_{d_{j,2}} \in \{-2, 0, 2\}$ be the sum of the two auxiliary variables. The zero-sum constraint requires $s_j + t_j = 0$, or equivalently, $s_j = -t_j$.

If the clause is satisfied, then $s_j\in\{-2,0,2\}$, and the auxiliary signs can be chosen so that $t_j=-s_j$. The block then contributes $9$ to the cut. If the clause is unsatisfied, then $s_j=\pm4$. The smallest possible value of $|s_j+t_j|$ is then $2$, so the largest cut contribution is $9-\frac14(2)^2=8$.

More precisely, let $\Phi$ denote the source instance, and let $\sigma$ be a Boolean assignment to its variables. Place each variable vertex on the side of the cut corresponding to its assigned value; in sign notation, set $z_{x_i}=2\sigma(x_i)-1$. The signs of the auxiliary vertices remain to be chosen. We say that a cut vector $z$ \emph{extends $\sigma$} if it has these prescribed signs on the variable vertices.

Let $\operatorname{sat}_{\Phi}(\sigma)$ denote the number of clauses satisfied by an assignment $\sigma$ to its variables. For a fixed assignment $\sigma$, optimizing the two auxiliary vertices in a $K_6$ block yields a contribution of $9$ if the corresponding clause is satisfied, and $8$ otherwise. Since the $K_6$ blocks are edge-disjoint and their auxiliary vertices are private to each clause, these choices can be made independently. Therefore,
\[
\max_{z \text{ extending } \sigma} \operatorname{MC}(G,z) = 9\operatorname{sat}_{\Phi}(\sigma) + 8\bigl(m-\operatorname{sat}_{\Phi}(\sigma)\bigr) = 8m+\operatorname{sat}_{\Phi}(\sigma).
\]
Define
\[
\operatorname{OPT}_{\mathrm{NAE}}(\Phi) = \max_{\sigma}\operatorname{sat}_{\Phi}(\sigma).
\]
Maximizing over $\sigma$ gives
\[
\operatorname{MC}(G) = 8m+\operatorname{OPT}_{\mathrm{NAE}}(\Phi).
\]
Together with $\operatorname{SDP}(G)=9m$, this yields the exact additive gap identity
\begin{equation}
\label{eq:exact-gap}
\operatorname{SDP}(G)-\operatorname{MC}(G) = m-\operatorname{OPT}_{\mathrm{NAE}}(\Phi).
\end{equation}
Thus the additive integrality gap equals the minimum number of unsatisfied clauses over all Boolean assignments. In particular, the SDP relaxation is exact if and only if $\Phi$ is satisfiable. This proves that recognizing SDP exactness is NP-hard for simple, unweighted graphs.

\begin{remark}[Connected instances]\label{rem:connected}
The graph need not be connected, but we can ensure connectedness while preserving both optimum values. First discard variables that occur in no clause, since their isolated vertices affect neither the Max-Cut nor the SDP optimum. If the constructed graph is disconnected, choose one local auxiliary vertex from each connected component and identify these vertices into a single vertex. The resulting graph is simple and connected.

Both optimal values are additive under this identification. Indeed, optimal cuts of the components can be complemented independently so that the identified vertices receive the same sign. Likewise, optimal SDP vector assignments can be placed in a common Euclidean space and transformed independently by orthogonal maps so that the identified vertices receive the same unit vector. These operations preserve all component objective values. Conversely, any cut or SDP assignment on the identified graph restricts to a feasible assignment on each original component. Thus the identification preserves both optima, including the SDP value $9m$, and hence preserves exactness. Consequently, recognizing SDP exactness remains NP-hard for connected simple, unweighted graphs.
\end{remark}

Together with Corollary~\ref{cor:np-membership}, this reduction and Remark~\ref{rem:connected} prove NP-completeness for connected simple unweighted graphs. Strong NP-completeness for nonnegative integer weights follows immediately, completing the proof of Theorem~\ref{thm:main}.

\begin{remark}
    Identity~\eqref{eq:exact-gap} concerns the exact additive gap. For an unsatisfiable source instance, it guarantees a gap of at least $1$, but does not establish a gap proportional to $m$. Kun's transformation preserves satisfiability; we do not use or establish preservation of a constant fraction of unsatisfied clauses.
\end{remark}
\begin{remark}
At this point, one may ask if we can use any other clique in this construction. Suppose we keep four variable vertices and add \(d\) private auxiliary vertices to form a clique. With \(d=0\), the four orthonormal variable vectors have a nonzero sum, and there are no auxiliary vectors to cancel it. Thus, the vector assignment used in our reduction does not attain the clique's SDP bound. At the other extreme, with even \(d \ge 4\), the auxiliary vertices can complete any variable assignment to a balanced cut, including an assignment placing all four variables on the same side. They therefore remove the intended constraint. The choice \(d=2\) provides exactly the needed freedom: the auxiliary vectors cancel the sum of the four orthonormal vectors, while the auxiliary signs complete a balanced cut precisely for NAE assignments. Finally, odd clique sizes cannot attain their SDP bounds by a cut, since \(\operatorname{MC}(K_q) = (q^2-1)/4 < q^2/4 = \operatorname{SDP}(K_q)\) for odd \(q \ge 3\).
\end{remark}
\subsection{Hardness with an optimal primal--dual pair supplied}
\begin{corollary}\label{cor:supplied-optimum}
Recognizing Max-Cut SDP exactness remains NP-complete for simple unweighted graphs when the input also includes an exact rational optimal primal--dual pair. The pair produced by the reduction has polynomial encoding length.
\end{corollary}
\begin{proof}
Use the graph before the identifications in Remark~\ref{rem:connected}. Let $b_j$ indicate the six vertices of block $j$, let $r_i$ count the blocks containing vertex $i$, and set
\[
M=\sum_{j=1}^m b_jb_j^\top.
\]
Edge-disjointness gives $M_{ij}=W_{ij}$ for $i\ne j$, $M_{ii}=r_i$, and $D_{ii}=5r_i$. Hence $y_i=6r_i$ satisfies
\[
\operatorname{Diag}(y)-L=M\succeq0,\qquad
\frac14\sum_i y_i=9m.
\]
The vector assignment in the reduction uses standard basis vectors for the variables and $-\frac12\sum_{i\in C_j}e_i$ for both auxiliaries of clause $j$. Its Gram matrix $X^*$ is rational, has polynomial encoding length, and attains $9m$. Thus $(X^*,y)$ is an exact optimal primal--dual pair, constructible in polynomial time regardless of satisfiability. Exactness still holds precisely when the source formula is satisfiable.

For membership in NP, the supplied pair can first be checked for primal feasibility, dual feasibility, and equality of objective values using exact rational arithmetic. A cut is then verified as in Corollary~\ref{cor:np-membership}. Equivalently, if validity of the supplied pair is treated as a promise, the cut verifier suffices.
\end{proof}
\begin{remark}
For the graph constructed in Corollary~\ref{cor:supplied-optimum}, let $b_j$ indicate the vertices of block $j$. The SDP optimal face is
\[
\mathcal F = \left\{ X\succeq0: \operatorname{diag}(X)=\mathbf1,\; Xb_j=0\ \text{for every }j \right\}.
\]
Indeed, a feasible matrix $X$ is optimal precisely when
\[
\langle M,X\rangle = \sum_j b_j^\top Xb_j =0.
\]
Since $X\succeq0$, each summand is nonnegative and vanishes if and only if $Xb_j=0$.

The construction provides an explicit rational point in $\mathcal F$ of polynomial encoding length. Nevertheless, deciding whether $\mathcal F$ contains a rank-one matrix is NP-complete for this family of instances. Such a matrix necessarily has the form $zz^\top$ with $z\in\{-1,1\}^{V}$, and belongs to $\mathcal F$ precisely when $b_j^\top z=0$ for every block. Thus the hardness persists even when the optimal face has an explicit description and a rational feasible point is supplied.
\end{remark}
The distinction illustrated here---between knowing the continuous optimum and deciding whether a cut attains it---also appears in Laurent and Poljak~\cite{laurent1996positive}.

\section{An Independent Bounded-Weight Construction}
\label{sec:weighted}
Theorem~\ref{thm:main} already implies strong NP-completeness of recognizing Max-Cut SDP exactness for nonnegative integer weights. We give a second proof that does not use the linear-hypergraph restriction or Kun's Sparse Incomparability Theorem \cite[Theorem~1]{kun2013constraints}. Its purpose is to isolate the local sum-of-squares mechanism underlying the hardness construction: each clause contributes a rank-one positive semidefinite matrix, and a private auxiliary vector makes its penalty vanish in the SDP. For sign assignments, simultaneous vanishing is equivalent to satisfying the source formula. Summing these local matrices produces nonnegative integer edge weights of polynomial magnitude. Thus the construction establishes strong NP-completeness directly, while separating the mechanism that enforces exactness from the additional structure needed to obtain simple unweighted graphs.

\subsection{The Base Combinatorial Problem}
Let \textnormal{\scshape Monotone $\{3,6\}$-in-9 SAT} be the Boolean satisfiability problem in which each clause contains exactly nine distinct, unnegated variables, and a clause is satisfied if and only if exactly three or exactly six of its variables are assigned true. Under the transformation $z_i = 2x_i - 1$ from Boolean variables to signs in $\{-1,1\}$, a clause $C_j$ is satisfied if and only if
\[
    \sum_{i \in C_j} z_i \in \{-3,3\}.
\]

\begin{lemma}
\textnormal{\scshape Monotone $\{3,6\}$-in-9 SAT} is NP-complete.
\end{lemma}

\begin{proof}
Define the Boolean relation
\[
    R = \left\{
        (x_1,\ldots,x_9) \in \{0,1\}^9 :
        \sum_{i=1}^{9} x_i \in \{3,6\}
    \right\}.
\]
We first allow variables to repeat within a constraint. By Schaefer's Dichotomy Theorem~\cite{schaefer1978complexity}, $\operatorname{CSP}(R)$ is NP-complete unless $R$ is 0-valid, 1-valid, Horn, dual-Horn, bijunctive, or affine. We exclude these possibilities:
\begin{itemize}
    \item \textit{Not 0-valid or 1-valid :}
    Neither the all-zero nor the all-one assignment belongs to $R$.

    \item \textit{Not bijunctive:}
    The three assignments with supports $\{1,2,3\}$, $\{4,5,6\}$, and $\{7,8,9\}$ belong to $R$. Their coordinatewise majority is the all-zero assignment, which does not belong to $R$. Thus $R$ is not closed
    under coordinatewise majority.

    \item \textit{Not Horn:}
    Two assignments with disjoint supports of size three belong to $R$, but their coordinatewise conjunction
    is the all-zero assignment. Thus $R$ is not closed under coordinatewise conjunction.

    \item \textit{Not dual-Horn:}
    The complements of those two assignments each have support of size six and hence belong to $R$.
    Their coordinatewise disjunction is the all-one assignment. Thus $R$ is not closed under
    coordinatewise disjunction.

    \item \textit{Not affine:}
    The relation has cardinality
    \[
        |R| = \binom{9}{3} + \binom{9}{6} = 168.
    \]
    Since a nonempty affine subspace over $\mathbb{F}_2$ has cardinality a power of two, $R$ is not affine.
\end{itemize}
It follows that $\operatorname{CSP}(R)$ is NP-complete.

We next show that requiring nine distinct variables in every constraint preserves NP-hardness. For distinct variables $x,y$, introduce eight fresh auxiliary variables $a_1,\ldots,a_8$ and impose
\[
    R(x,a_1,\ldots,a_8)
    \quad\text{and}\quad
    R(y,a_1,\ldots,a_8).
\]
This gadget enforces $x=y$. Indeed, writing $t=\sum_{i=1}^{8}a_i$, its constraints require
\[
    t+x \in \{3,6\}
    \quad\text{and}\quad
    t+y \in \{3,6\}.
\]
If $x \ne y$, these two quantities differ by one, which is impossible. Conversely, either assignment with $x=y$ extends to a satisfying assignment of the gadget: choose $t=3$ when $x=y=0$, and $t=2$ when $x=y=1$.

Given an arbitrary instance of $\operatorname{CSP}(R)$, replace each argument occurrence in each constraint by a fresh variable, and enforce equality between that variable and its original variable using the gadget above. Use fresh auxiliary variables for each gadget. Every resulting constraint contains nine distinct,
unnegated variables. The construction preserves satisfiability and increases the instance size only linearly.
Thus the distinct-variable restriction is NP-hard.

Finally, a Boolean assignment can be checked against all clauses in polynomial time, so the problem belongs to NP.
\end{proof}

\subsection{The Matrix Construction and Its Gap}

Let $\Phi$ be a formula with $n$ variables and $m\ge1$ clauses. We construct a weighted graph with one vertex for each variable and one private auxiliary vertex $y_j$ for each clause $C_j$. The purpose of the auxiliary vertex is to turn the clause condition $\sum_{i\in C_j}z_i\in\{-3,3\}$ into a zero-sum condition:
\[
\sum_{i\in C_j}z_i+3z_{n+j}=0,
\]
where $z_{n+j}\in\{-1,1\}$ is the sign assigned to $y_j$.

Encode this expression by a vector $a_j\in\mathbb Z^{n+m}$, with $(a_j)_i=1$ for $i\in C_j$, $(a_j)_{n+j}=3$, and all other entries zero. Set
\[
M=\sum_{j=1}^m a_ja_j^\top, \qquad W_{uv}=M_{uv}\quad(u\ne v), \qquad W_{uu}=0.
\]
The off-diagonal entries define the edge weights of a graph $G$ on $n+m$ vertices. Variable--variable weights count shared clauses, variable--auxiliary weights are $3$ within the corresponding clause, and auxiliary--auxiliary weights vanish. In particular, all weights are nonnegative integers bounded by $3m$. Clauses may share several variables: their contributions to a common edge are simply added.

For any feasible Gram matrix $X$, the Max-Cut SDP objective satisfies
\begin{equation}\label{eq:weighted-sos}
\begin{aligned}
\frac14\sum_{u\ne v}W_{uv}(1-X_{uv}) &= \frac14\bigl( \mathbf1^\top M\mathbf1-\langle M,X\rangle \bigr)\\
&= 36m-\frac14\langle M,X\rangle.
\end{aligned}
\end{equation}
The first equality holds because $X_{uu}=1$, so the diagonal terms cancel. For the second, each $a_j$ has coordinate sum $12$, and hence
\[
\mathbf1^\top M\mathbf1 = \sum_{j=1}^m(\mathbf1^\top a_j)^2 = 144m.
\]
Since $M,X\succeq0$, we have $\langle M,X\rangle\ge0$, giving the upper bound $\operatorname{SDP}(G)\le36m$.

We next construct a vector assignment attaining this bound for every source formula. Assign the variable vertices orthonormal vectors $v_1,\ldots,v_n$ and assign
\[
u_j=-\frac13\sum_{i\in C_j}v_i
\]
to the auxiliary vertex $y_j$. Each clause contains nine distinct variables, so
\[
\|u_j\|^2=\frac19\left\|\sum_{i\in C_j}v_i\right\|^2=1.
\]
Thus all assigned vectors are unit vectors. Moreover, the resulting Gram matrix satisfies
\[
\langle M,X\rangle = \sum_{j=1}^m a_j^\top Xa_j = \sum_{j=1}^m \left\|\sum_{i\in C_j}v_i+3u_j\right\|^2 = 0.
\]
Consequently,
\[
\operatorname{SDP}(G)=36m.
\]

For a cut vector $z\in\{-1,1\}^{n+m}$, the same matrix expression becomes a sum of scalar squares:
\[
\langle M,zz^\top\rangle = \sum_{j=1}^m \left(\sum_{i\in C_j}z_i+3z_{n+j}\right)^2.
\]
Evaluating~\eqref{eq:weighted-sos} on cut matrices therefore gives the exact gap identity
\begin{equation}\label{eq:weighted-gap}
\operatorname{SDP}(G)-\operatorname{MC}(G) = \frac14 \min_{z\in\{-1,1\}^{n+m}} \sum_{j=1}^m \left(\sum_{i\in C_j}z_i+3z_{n+j}\right)^2.
\end{equation}

If a cut attains the SDP bound, every square must vanish. For each clause, this forces $\sum_{i\in C_j}z_i\in\{-3,3\}$, so the variable signs give a satisfying assignment of $\Phi$. Conversely, a satisfying assignment has each clause sum equal to $3$ or $-3$. Its private auxiliary sign can then be chosen to cancel that sum. These choices can be made independently for all clauses, giving a cut of value $36m$.

Thus $G$ is exact if and only if $\Phi$ is satisfiable. The construction has polynomial size and produces nonnegative integer weights of polynomial magnitude, proving strong NP-hardness independently of the unweighted reduction. Together with Corollary~\ref{cor:np-membership}, this gives an independent proof of strong NP-completeness.

\begin{remark}
This construction fits the Gram-representation framework of Delorme and Poljak~\cite{DP93_II}. Associate with vertex $u$ the feature vector
\[
f_u=((a_1)_u,\ldots,(a_m)_u).
\]
Then $W_{uv}=\langle f_u,f_v\rangle$ for $u\ne v$. The upper bound in~\eqref{eq:weighted-sos} is
\[
\frac14\left\|\sum_u f_u\right\|^2=36m,
\]
and the vector assignment above certifies its attainment by the SDP.
\end{remark}

\section{Extensions to Other Partitioning Problems}
\label{sec:exactness-applications}
We first extend the clause construction to Max-$k$-Cut, then give reductions for Max-DiCut and Max-Bisection. Exactness throughout refers to the particular SDP formulation displayed in each subsection. Table~\ref{tab:exactness-applications} summarizes the results.
\begin{table}[htbp]
\centering\small
\begin{tabular}{@{}>{\raggedright\arraybackslash}p{.16\textwidth}>{\raggedright\arraybackslash}p{.52\textwidth}>{\raggedright\arraybackslash}p{.25\textwidth}@{}}
\toprule
Problem & Restricted instances & Exactness recognition\\
\midrule
Max-$k$-Cut & Connected; nonnegative integer weights; fixed $k\ge3$ & Strongly NP-hard\\[3mm]
Max-DiCut & Strongly connected; positive integer arc weights; equal weighted indegree and outdegree at every vertex; unequal opposite arc weights & Strongly NP-complete for~\eqref{eq:dicut-sdp}\\[1.2cm]
Max-Bisection & Connected, simple, unweighted & NP-hard\\
\bottomrule\\
\end{tabular}
\caption{Results for the SDP formulations in this section.}
\label{tab:exactness-applications}
\end{table}

The NP-completeness statements include polynomial-size certificates of exactness. For the basic Max-DiCut relaxation,
such a certificate follows from Lemma~\ref{lem:rank-one-certificate}. For the Max-$k$-Cut and Max-Bisection formulations, we establish NP-hardness only; the additional constraints prevent a direct application of that lemma, and we do not establish NP membership here.

\subsection{Max-\texorpdfstring{$k$}{k}-Cut}
\label{subsec:max-k-cut-exactness}
We extend the clause construction to Max-$k$-Cut for every fixed $k\ge3$. View a partition into at most $k$ parts as an assignment of $k$ available colors to the vertices. To encode the Boolean constraints, we introduce $k-2$ anchor vertices shared by all clauses. The edge weights ensure that any partition attaining the SDP upper bound assigns distinct colors to the anchors and uses only the two remaining colors on the variable and auxiliary vertices. We then construct a feasible vector assignment attaining this bound in the Frieze--Jerrum relaxation.

For a graph $G$ with nonnegative edge weights $W$, let $\operatorname{MC}_k(G)$ denote the maximum weight of edges joining different parts of a partition into at most $k$ parts. The Frieze--Jerrum SDP relaxation \cite{FriezeJerrum1997} is
\begin{equation}
\label{eq:fj-sdp-exactness}
\begin{aligned}
\operatorname{SDP}_k(G) ={}&\max_X\ \frac{k-1}{k}\sum_{u<v}W_{uv}(1-X_{uv})\\
\text{subject to }&X\succeq0,\qquad X_{uu}=1\quad(u\in V(G)),\\
&X_{uv}\ge-\frac1{k-1}\quad(u\ne v).
\end{aligned}
\end{equation}
A discrete partition is represented by assigning its parts the unit vectors of a regular simplex with mutual inner products $-1/(k-1)$. For such an assignment, the SDP objective equals the partition weight. Exactness means $\operatorname{SDP}_k(G)=\operatorname{MC}_k(G)$.

\begin{proof}[Proof of Theorem~\ref{thm:kcut-main}]
Let $\Phi$ be a linear Monotone NAE-4-SAT instance with $m\ge1$ clauses $C_1, C_2, \allowbreak \ldots,C_m$, using the NP-complete restriction in Lemma~\ref{lem:linear-nae4}. Discard variables occurring in no clause. For each clause $C_j$, introduce two private auxiliary vertices $d_{j,1},d_{j,2}$ and set
\[
B_j=C_j\cup\{d_{j,1},d_{j,2}\}.
\]
Let $V_0$ consist of the variable and auxiliary vertices. Add $k-2$ anchor vertices $a_1,\ldots,a_{k-2}$, shared by all clauses, and write $V=V_0\cup\{a_1,\ldots,a_{k-2}\}$.

For each clause define $b_j\in\mathbb{Z}^{V}$ by
\[
(b_j)_v=
\begin{cases}
1,&v\in B_j,\\
3,&v\in\{a_1,\ldots,a_{k-2}\},\\
0,&\text{otherwise}.
\end{cases}
\]
Construct
\[
M=\sum_{j=1}^m b_jb_j^\top, \qquad W_{uv}=M_{uv}\quad(u\ne v), \qquad W_{uu}=0.
\]
These weights define a graph $G$ on $V$. Each $B_j$ induces a $K_6$. Since $\Phi$ is linear and each auxiliary vertex belongs to only one block, an edge within $V_0$ has weight $1$ whenever it is present. The weight between an anchor and $v\in V_0$ is three times the number of blocks containing $v$, and the weight between two anchors is $9m$. Thus every edge weight is a positive integer bounded by $9m$. Every vertex in $V_0$ is adjacent to every anchor, so $G$ is connected.

\smallskip
\noindent\emph{The SDP upper bound.}
For a feasible matrix $X$ in \eqref{eq:fj-sdp-exactness}, let $F_k(X)$ denote its objective value. Using $X_{uu}=1$ gives
\begin{align*}
F_k(X) &=\frac{k-1}{2k} \left(\sum_{u\ne v}M_{uv}-\sum_{u\ne v}M_{uv}X_{uv}\right)\\
&=\frac{k-1}{2k} \left(\mathbf1^\top M\mathbf1-\langle M,X\rangle\right).
\end{align*}
Since $\mathbf1^\top b_j=6+3(k-2)=3k$, it follows that
\begin{equation}
\label{eq:fj-sos-upper-bound}
F_k(X)=U-\frac{k-1}{2k}\langle M,X\rangle, \qquad U=\frac92 k(k-1)m.
\end{equation}
Both $M$ and $X$ are positive semidefinite, so $F_k(X)\le U$.

\smallskip
\noindent\emph{An optimal vector assignment.}
First assign mutually orthogonal unit vectors $u_v$ to the variable vertices. For each clause set
\[
u_{d_{j,1}}=u_{d_{j,2}} =-\frac12\sum_{v\in C_j}u_v.
\]
The four variables in a clause are distinct, so these auxiliary vectors have unit norm. Moreover,
\begin{equation}
\label{eq:fj-original-block-balance}
\sum_{v\in B_j}u_v=0\qquad(j=1,\ldots,m).
\end{equation}

Choose unit vectors $q_1,\ldots,q_{k-2}$ with $\langle q_r,q_s\rangle=-1/(k-1)$ for $r\ne s$, for example a subset of the vertices of a regular simplex with $k$ vertices. Choose their span orthogonal to all the vectors $u_v$. We modify the vectors on $V_0$ by a common translation and scaling. The translation will cancel the contribution of the anchors in every clause, while the scaling will
restore unit norms. Set
\[
Q=\sum_{r=1}^{k-2}q_r, \qquad c=-\frac12Q, \qquad \alpha=\sqrt{\frac{k}{2(k-1)}}.
\]
Assign $w_v=c+\alpha u_v$ to each $v\in V_0$ and assign $q_r$ to anchor $a_r$. Direct calculation gives
\[
\|Q\|^2=\frac{2(k-2)}{k-1}, \qquad \|c\|^2=\frac{k-2}{2(k-1)}, \qquad \langle c,q_r\rangle=-\frac1{k-1}.
\]
Hence each $w_v$ is a unit vector and $\langle w_v,q_r\rangle=-1/(k-1)$. For $v,w\in V_0$,
\begin{equation}
\label{eq:fj-affine-embedding}
\langle w_v,w_w\rangle =\frac{k-2}{2(k-1)} +\frac{k}{2(k-1)}\langle u_v,u_w\rangle \ge-\frac1{k-1}.
\end{equation}
The anchor pairs also satisfy the required lower bound. The resulting Gram matrix $X$ is therefore feasible for \eqref{eq:fj-sdp-exactness}.

For every clause, \eqref{eq:fj-original-block-balance} implies
\[
\sum_{v\in B_j}w_v+3\sum_{r=1}^{k-2}q_r =6c+\alpha\sum_{v\in B_j}u_v+3Q=0.
\]
Consequently,
\[
\langle M,X\rangle =\sum_{j=1}^m \left\|\sum_{v\in B_j}w_v+3\sum_{r=1}^{k-2}q_r\right\|^2 =0.
\]
Together with \eqref{eq:fj-sos-upper-bound}, this proves
\begin{equation}
\label{eq:fj-known-optimum}
\operatorname{SDP}_k(G)=U=\frac92 k(k-1)m.
\end{equation}

\smallskip
\noindent\emph{Discrete attainment of the bound.}
Let $\sigma:V\to\{1,\ldots,k\}$ be a coloring representing a partition, and write $\operatorname{cut}_k(G,\sigma)$ for the total weight of edges whose endpoints receive different colors. Let $p_1,\ldots,p_k$ be unit regular-simplex vectors with mutual inner products $-1/(k-1)$. Define the coefficient load of color $\ell$ in clause $j$ by
\[
L_{j,\ell}(\sigma) =\sum_{\substack{v\in V\\\sigma(v)=\ell}}(b_j)_v.
\]
Thus $L_{j,\ell}(\sigma)$ counts the vertices of $B_j$ assigned color $\ell$, together with a contribution of $3$ from each anchor assigned that color. The loads sum to $6+3(k-2)=3k$. The simplex inner products therefore give
\begin{align*}
\left\|\sum_{\ell=1}^kL_{j,\ell}(\sigma)p_\ell\right\|^2 &=\frac{k}{k-1}\sum_{\ell=1}^kL_{j,\ell}(\sigma)^2 -\frac1{k-1}(3k)^2\\
&=\frac{k}{k-1}\sum_{\ell=1}^k \bigl(L_{j,\ell}(\sigma)-3\bigr)^2.
\end{align*}
Evaluating \eqref{eq:fj-sos-upper-bound} at the Gram matrix of this partition yields
\begin{equation}
\label{eq:fj-discrete-penalty}
\operatorname{cut}_k(G,\sigma) =U-\frac12\sum_{j=1}^m\sum_{\ell=1}^k \bigl(L_{j,\ell}(\sigma)-3\bigr)^2.
\end{equation}
Thus a partition attains $U$ if and only if every color has load $3$ in every clause.

Each anchor contributes $3$ to the load of its color in every clause. Accordingly, no two anchors can have the same color in a partition attaining $U$, since that color would have load at least $6$. Furthermore, no vertex of $V_0$ can use an anchor color: it belongs to some block and would increase the corresponding load above $3$. Only two colors remain for $V_0$, and each six-vertex block must contain three vertices of each color. Figure~\ref{fig:kcut-clause-loads} illustrates this load condition for $k=4$ and a single clause.

\begin{figure}[htbp]
\centering
\begin{tikzpicture}[
    vertex/.style={
        circle,
        draw=black,
        fill=black!20,
        minimum size=8mm,
        inner sep=1pt,
        font=\small
    },
    auxiliary/.style={
        rectangle,
        draw=black,
        fill=white,
        minimum width=10mm,
        minimum height=7mm,
        inner sep=2pt,
        font=\small
    },
    anchor vertex/.style={
        circle,
        draw=black,
        double,
        double distance=1pt,
        fill=white,
        minimum size=8mm,
        inner sep=1pt,
        font=\small
    },
    partition class/.style={
        rounded corners=3pt,
        draw=black!55,
        line width=0.6pt
    }
]

% Four boxes represent the four partition classes.
\foreach \pos/\c in {-3.6/1,-1.2/2,1.2/3,3.6/4} {
    \begin{scope}[xshift=\pos cm]
        \draw[partition class] (-1,-1.65) rectangle (1,1.65);
        \node[font=\small] at (0,2.05) {Color \c};
        \node[font=\small] at (0,-2.05)
            {$L_{j,\c}=3$};
    \end{scope}
}

% Each anchor contributes coefficient 3 to every clause.
\node[anchor vertex] at (-3.6,0.25) {$a_1$};
\node[font=\small,align=center] at (-3.6,-0.65)
    {coefficient $3$};

\node[anchor vertex] at (-1.2,0.25) {$a_2$};
\node[font=\small,align=center] at (-1.2,-0.65)
    {coefficient $3$};

% An illustrative balanced split of the six vertices in B_j.
% Each of these vertices has coefficient 1 in b_j.
\node[vertex]  at (1.2, 1.05) {$x_{j_1}$};
\node[vertex]  at (1.2, 0)    {$x_{j_2}$};
\node[auxiliary] at (1.2,-1.05) {$d_{j,1}$};

\node[vertex]  at (3.6, 1.05) {$x_{j_3}$};
\node[vertex]  at (3.6, 0)    {$x_{j_4}$};
\node[auxiliary] at (3.6,-1.05) {$d_{j,2}$};

\end{tikzpicture}
\caption{\small The clause-load condition for $k=4$.
The boxes represent the four parts of a partition
attaining the SDP bound. Each of the two shared
anchors contributes $3$ to its color's load in every
clause, leaving no capacity for a block vertex in
either anchor color. The six vertices of
$B_j=C_j\cup\{d_{j,1},d_{j,2}\}$ each contribute $1$
and must split equally between the two remaining
colors. The figure shows one such split.
Only vertices with nonzero coefficient in $b_j$
are shown; vertices from other blocks may also
occupy the last two parts.}
\label{fig:kcut-clause-loads}
\end{figure}
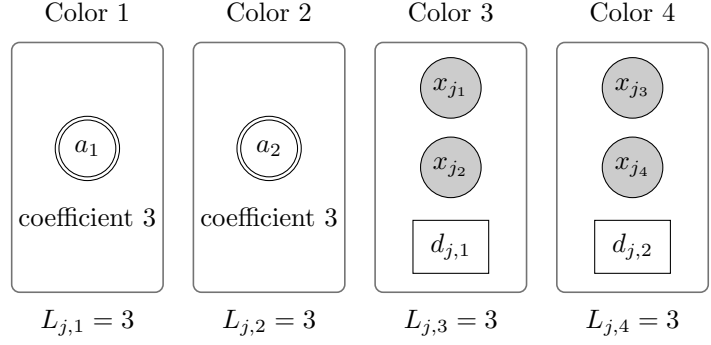

For four variable vertices and two private auxiliary vertices, such a three-versus-three split exists precisely when the four variables are not monochromatic. Indeed, if the variables use one color $t$ times, then the auxiliary vertices can complete the split exactly when $t\in\{1,2,3\}$. Hence a partition attaining $U$ induces a satisfying assignment of $\Phi$. Conversely, a satisfying assignment of $\Phi$ extends to a partition attaining $U$: use its two Boolean values as the two non-anchor colors, assign distinct remaining colors to the anchors, and balance each block using its private auxiliary vertices.

It follows that
\[
\operatorname{MC}_k(G)=\operatorname{SDP}_k(G) \quad\Longleftrightarrow\quad \Phi\text{ is satisfiable}.
\]
For fixed $k$, the construction has polynomial size and all weights are bounded by $9m$. This proves strong NP-hardness.
\end{proof}

Combining \eqref{eq:fj-known-optimum} and \eqref{eq:fj-discrete-penalty} gives the exact additive integrality gap:
\begin{equation}
\label{eq:fj-exact-gap}
\operatorname{SDP}_k(G)-\operatorname{MC}_k(G) =\frac12\min_{\sigma:V\to\{1,\ldots,k\}} \sum_{j=1}^m\sum_{\ell=1}^k \bigl(L_{j,\ell}(\sigma)-3\bigr)^2.
\end{equation}
The continuous assignment balances every clause regardless of satisfiability, while discrete balance forces the anchors to occupy $k-2$ distinct colors and recovers the original Boolean constraints. Thus the reduction preserves the central feature of the earlier constructions: the SDP optimum and an optimal vector assignment are explicit, but recognizing whether a discrete partition attains that optimum is strongly NP-hard.
\begin{remark}
    The load identity \eqref{eq:fj-discrete-penalty} is a specialization of the feature-balancing formulation of Bhardwaj, Gogoi, and Narayanan~\cite{BGN26}. In the present construction, each vertex \(v\) has feature vector \(((b_1)_v,\ldots,(b_m)_v)\), and attaining \(U\) requires equal aggregate feature vectors in all \(k\) parts. The shared anchors encode the Boolean constraints within this balancing condition, while the explicit vector assignment attains the SDP bound independently of satisfiability.
\end{remark}

\subsection{Max-DiCut}
\label{subsec:dicut-exactness}
Let $D$ be a digraph with nonnegative arc weights $W_{ij}$ and no loops. For a sign vector $x\in\{-1,1\}^{V}$, the directed cut from the positive part to the negative part has weight
\[
\operatorname{dcut}_D(x) =\frac14\sum_{i\ne j}W_{ij}(1+x_i-x_j-x_ix_j).
\]
We denote its maximum by $\operatorname{OPT}_{\mathrm{di}}(D)$. To express the linear terms in $x$ through a Gram matrix, introduce an additional index $0$ representing the fixed sign $x_0=1$. Then $x_i=x_ix_0$, so the products in the objective can be replaced by entries of a correlation matrix. This gives the SDP relaxation
\begin{equation}
\label{eq:dicut-sdp}
\begin{aligned}
\operatorname{SDP}_{\mathrm{di}}(D) = {}&\max_X\, \frac14\sum_{i\ne j}W_{ij} (1+X_{i0}-X_{j0}-X_{ij})\\
\text{subject to }&X\succeq0,\qquad X_{uu}=1\quad(u\in V\cup\{0\}).
\end{aligned}
\end{equation}

The discrete assignment $x$ corresponds to $X=(1,x)(1,x)^\top$.

Replacing each undirected edge by two opposite arcs, each carrying the original edge weight, preserves both the maximum cut value and the SDP optimum. To obtain unequal opposite arc weights, we perturb this symmetric assignment. The perturbation adds weight in one direction and subtracts it in the other, while preserving equal weighted indegree and outdegree at every vertex. The following lemma shows that such a perturbation leaves both objective values unchanged, apart from a common scaling factor.

\begin{lemma}
\label{lem:circulation-invariance}
Let $G$ have a symmetric nonnegative weight matrix $A$ with zero diagonal. Suppose $J^\top=-J$, $J\mathbf1=0$, and $C>0$ is such that $W=CA+J$ is entrywise nonnegative. Let $D$ have arc-weight matrix $W$. Then
\begin{align}
\operatorname{SDP}_{\mathrm{di}}(D) &=C\operatorname{SDP}(G), \label{eq:dicut-sdp-transfer}\\
\operatorname{OPT}_{\mathrm{di}}(D) &=C\operatorname{MC}(G). \label{eq:dicut-integer-transfer}
\end{align}
\end{lemma}

\begin{proof}
Symmetry of $A$, skew-symmetry of $J$, and $J\mathbf1=0$ imply $W\mathbf1=W^\top\mathbf1$. Thus the weighted outdegree and indegree agree at every vertex, and the anchor terms cancel:
\[
\sum_{i\ne j}W_{ij}(X_{i0}-X_{j0}) =\sum_i X_{i0}\sum_{j\ne i}(W_{ij}-W_{ji})=0.
\]
Moreover, $\sum_{i\ne j}J_{ij}=0$ and $\sum_{i\ne j}J_{ij}X_{ij}=0$ because $X$ is symmetric. Hence the objective in \eqref{eq:dicut-sdp} equals
\[
\frac C4\sum_{i\ne j}A_{ij}(1-X_{ij}).
\]
The principal submatrix $X_V$ is feasible for the Max-Cut SDP on $G$. Conversely, every feasible Max-Cut Gram matrix $Y$ extends to a feasible anchored matrix $\begin{pmatrix}1&0\\0&Y\end{pmatrix}$, obtained by choosing the anchor orthogonal to all variable vectors. This proves \eqref{eq:dicut-sdp-transfer}. Evaluating the same identity at $X=(1,x)(1,x)^\top$ and maximizing over $x$ proves \eqref{eq:dicut-integer-transfer}.
\end{proof}

\begin{theorem}
\label{thm:dicut-exactness-hardness}
Recognizing exactness of Max-DiCut SDP relaxation \eqref{eq:dicut-sdp} is strongly NP-complete even for strongly connected digraphs with nonnegative integer arc weights and equal weighted indegree and outdegree at every vertex such that, for every adjacent pair $i,j$, both $W_{ij}$ and $W_{ji}$ are positive and $W_{ij}\ne W_{ji}$.
\end{theorem}

\begin{proof}
We first prove membership in NP for the basic relaxation~\eqref{eq:dicut-sdp}. Write its objective as $c+\langle B,X\rangle$, where
\[
c=\frac14\sum_{i\ne j}W_{ij},\qquad
B_{ij}=-\frac{W_{ij}+W_{ji}}8\quad(i,j\in V,\ i\ne j),
\]
\[
B_{i0}=B_{0i}=\frac18\sum_{j\ne i}(W_{ij}-W_{ji}),\qquad B_{uu}=0.
\]
These rational coefficients have polynomial encoding length. A rank-one feasible matrix has the form $zz^\top$ for a sign vector $z$; replacing $z$ by $z_0z$ ensures $z_0=1$ and identifies a directed cut. Therefore exactness is equivalent to the existence of a rank-one optimum, which has a polynomial-time certificate by Lemma~\ref{lem:rank-one-certificate}.

For hardness, take a connected simple unweighted graph $G$ from the $K_6$ construction above, using the vertex identification in Remark~\ref{rem:connected} to ensure connectedness. Every edge still belongs to a $K_6$, so $G$ is bridgeless. Let $A$ be its adjacency matrix and let $r=|E(G)|$.

By Robbins' theorem \cite{Robbins1939}, $G$ admits a strongly connected orientation. Such an orientation can be constructed in polynomial time: orient the tree edges of a depth-first search away from the root and the remaining edges toward their ancestors. For each oriented edge $u\to v$, choose a directed path from $v$ to $u$ without repeated vertices. Together with $u\to v$, this path gives a directed cycle. Superimpose unit circulations on these $r$ cycles. The resulting integer flow $f$ satisfies conservation at every vertex and
\[
1\le f(e)\le r \qquad\text{for every oriented edge }e.
\]
Define $J_{uv}=f(u\to v)$ and $J_{vu}=-f(u\to v)$ for each oriented edge, and set all remaining entries to zero. Then $J^\top=-J$, $J\mathbf1=0$, and $J$ is nonzero on every edge of $G$.

Set $C=r+1$ and $W=CA+J$. On every edge of $G$, the two arc weights are $C+f(e)$ and $C-f(e)$. They are distinct positive integers bounded by $2r+1$. The support contains both orientations of every edge of the connected graph $G$, so $D$ is strongly connected. Its weighted indegree and outdegree agree at every vertex by construction.

Lemma~\ref{lem:circulation-invariance} gives
\begin{equation}
\label{eq:dicut-gap-transfer}
\operatorname{SDP}_{\mathrm{di}}(D) -\operatorname{OPT}_{\mathrm{di}}(D) =C\bigl(\operatorname{SDP}(G) -\operatorname{MC}(G)\bigr).
\end{equation}
Thus $D$ is exact if and only if $G$ is exact. The construction takes polynomial time and produces polynomially bounded integer weights. The unweighted Max-Cut exactness hardness result in Theorem~\ref{thm:main} therefore implies strong NP-hardness. Together with membership in NP, this proves the theorem.
\end{proof}

The hardness argument also applies if \eqref{eq:dicut-sdp} is strengthened by the local consistency inequalities on the anchor and each pair of variables, as in the Boolean two-variable CSP relaxation \cite{BrakensiekEtAl2023}:
\[
1+sX_{i0}+tX_{j0}+stX_{ij}\ge0 \qquad(s,t\in\{-1,1\}).
\]
Indeed, the extension $\begin{pmatrix}1&0\\0&Y\end{pmatrix}$ used in the lemma satisfies these inequalities because they reduce to $1\pm Y_{ij}\ge0$. Hence the value identity and strong NP-hardness remain valid for this strengthened relaxation. The NP-membership argument above concerns the basic relaxation~\eqref{eq:dicut-sdp}.

\subsection{Max-Bisection}
\label{subsec:bisection-exactness}
A bisection partitions the vertices of a graph into two sets of equal size. For a graph $H$ with an even number of vertices, let $\operatorname{OPT}_{\mathrm{bis}}(H)$ denote the maximum weight of edges crossing such a partition. In sign notation, the equal-size condition is $\mathbf1^\top z=0$. We consider the SDP relaxation
\begin{equation}
\label{eq:bisection-sdp}
\operatorname{SDP}_{\mathrm{bis}}(H) =\max\left\{ \frac14\sum_{u\ne v}W^H_{uv}(1-X_{uv}): X\succeq0,\ X_{uu}=1,\ \mathbf1^\top X\mathbf1=0 \right\}.
\end{equation}
For a Gram representation $X_{uv}=\langle v_u,v_v\rangle$, the last constraint is equivalent to $\sum_u v_u=0$.

The familiar reduction from Max-Cut to Max-Bisection takes two disjoint copies of the input graph \cite{RaghavendraTan2012}. We connect these copies by a matching between corresponding vertices. This graph construction also appears in other reduction settings~\cite{FellowsEtAl2012}. Here we assign opposite vectors $v_i$ and $-v_i$ to corresponding vertices. These antipodal assignments satisfy the balance constraint and give an exact relation between the two SDP values.

\begin{lemma}
\label{lem:antipodal-extension}
Let $G$ be a graph on $n$ vertices with nonnegative edge weights. Construct $H$ from two copies of $G$, retaining their edge weights, and add an edge of weight $1$ between each vertex $i$ and its copy $\bar i$. Then
\begin{align}
\operatorname{SDP}_{\mathrm{bis}}(H) &=2\operatorname{SDP}(G)+n, \label{eq:bisection-sdp-transfer}\\
\operatorname{OPT}_{\mathrm{bis}}(H) &=2\operatorname{MC}(G)+n. \label{eq:bisection-integer-transfer}
\end{align}
\end{lemma}

\begin{proof}
Restricting any feasible vector assignment on $H$ to either copy of $G$ gives a feasible Max-Cut vector assignment. Each copy therefore contributes at most $\operatorname{SDP}(G)$, while the $n$ matching edges contribute at most $n$. This proves the upper bound in \eqref{eq:bisection-sdp-transfer}.

Let $v_1,\ldots,v_n$ attain the Max-Cut SDP optimum on $G$. Assign $v_i$ to $i$ and $-v_i$ to $\bar i$. Both copies attain the original SDP value, every matching edge contributes $1$, and
\[
\sum_{i=1}^n(v_i-v_i)=0.
\]
The assignment is consequently feasible for \eqref{eq:bisection-sdp} and attains the upper bound. Equivalently, if $Y$ is the original optimal Gram matrix, the extension is
\[
X=\begin{pmatrix}Y&-Y\\-Y&Y\end{pmatrix}, \qquad X\succeq0,\quad X\mathbf1=0.
\]

In the discrete problem, each copy contributes at most $\operatorname{MC}(G)$ and the matching contributes at most $n$. If $x$ is a maximum cut of $G$, the assignment $(x,-x)$ is a bisection of $H$ and attains all three bounds. This proves \eqref{eq:bisection-integer-transfer}.
\end{proof}

\begin{theorem}
\label{thm:bisection-exactness-hardness}
Recognizing exactness of Max-Bisection SDP relaxation \eqref{eq:bisection-sdp} is NP-hard even for connected simple unweighted graphs. Consequently, it is strongly NP-hard for graphs with nonnegative integer edge weights.
\end{theorem}

\begin{proof}
Apply Lemma~\ref{lem:antipodal-extension} to a connected simple unweighted instance $G$ from the Max-Cut exactness reduction, using Remark~\ref{rem:connected} to ensure connectedness. The resulting graph $H$ is connected, simple, and unweighted, and the construction takes polynomial time. Subtracting the two identities in the lemma yields
\begin{equation}
\label{eq:bisection-gap-transfer}
\operatorname{SDP}_{\mathrm{bis}}(H) -\operatorname{OPT}_{\mathrm{bis}}(H) =2\bigl(\operatorname{SDP}(G) -\operatorname{MC}(G)\bigr).
\end{equation}
Thus $H$ is exact if and only if $G$ is exact, proving the result.
\end{proof}

The identity~\eqref{eq:bisection-sdp-transfer} also describes all optimal SDP solutions on $H$. The two copies and the matching each have an upper bound on their contribution, and the SDP optimum equals the sum of these bounds. Every optimal solution must therefore attain all of them. In particular, $X_{i\bar i}=-1$ for every matching edge. Since the corresponding vectors have unit norm,
they must be opposites. Their contributions to the global vector sum cancel pairwise, and the restriction to either copy is an optimal Max-Cut SDP solution on $G$. Conversely, every optimal solution on $G$ extends in this way to an optimal solution on $H$.

\section{Summary and Conclusion}
\label{sec:conclusion}
In this manuscript, we resolve the complexity of recognizing exactness of the standard Max-Cut SDP relaxation for simple unweighted graphs. The problem is NP-complete even when the graph is connected, and hence strongly NP-complete for nonnegative integer edge weights. This answers the unweighted recognition question posed by Delorme and Poljak and shows that the difficulty of recognition persists without large numerical weights.

The reduction also identifies precisely where this difficulty arises. For the constructed instances, the SDP optimum is known explicitly, and an exact rational optimal primal--dual pair can be produced in polynomial time, independently of whether the source formula is satisfiable. Nevertheless, deciding whether a cut attains this certified optimum remains NP-complete. Thus, unless $\mathrm{P}=\mathrm{NP}$, there is no polynomial-time algorithm for recognizing exactness even when an exact optimal primal--dual pair is supplied. In geometric terms, the construction supplies a rational point in an explicitly described optimal face, while the existence of a rank-one point in that face encodes the underlying satisfiability problem.

The clause construction gives more than an equivalence between satisfiability and exactness: its additive integrality gap equals the minimum number of unsatisfied clauses in the linear source instance. Each violated clause contributes exactly one unit to the gap after its auxiliary vertices are optimized. This identity explains how the combinatorial obstruction to satisfiability appears in the relaxation gap. It does not, however, establish a constant relative gap, since the transformation to linear instances is used only to preserve satisfiability. The independent bounded-weight construction provides a second proof of strong NP-completeness without Kun's theorem, showing directly how sums of clause-specific squares enforce the required Boolean constraints.

The extensions show that this recognition difficulty also occurs for other semidefinite relaxations of partitioning problems. For the Frieze--Jerrum Max-$k$-Cut relaxation, strong NP-hardness holds for every fixed $k\ge3$, even on connected graphs with nonnegative integer edge weights. Here too, the construction provides an explicit SDP optimum that a discrete partition attains if and only
if the source formula is satisfiable. For the basic Max-DiCut relaxation, recognition is strongly NP-complete even with balanced weighted degrees and positive unequal opposite arc weights. For the Max-Bisection relaxation considered here, recognition is NP-hard even on connected simple unweighted graphs. The latter two reductions preserve the Max-Cut additive integrality gap up to an explicit factor.

These hardness results complement structural criteria for exactness. For the standard Max-Cut SDP, a given cut determines a rational slack matrix whose positive semidefiniteness certifies that the cut is SDP-optimal. Consequently, exactness can be recognized in polynomial time on any graph class where a maximum cut can be found in polynomial time. The general recognition problem is difficult despite this efficient verification of a proposed exact cut.

An important direction is to determine whether comparable recognition hardness holds for stronger semidefinite relaxations, including fixed levels of sum-of-squares hierarchies. The present hardness results do not automatically extend to these formulations: additional constraints may exclude the vector assignments that certify the SDP values in our reductions. Extending our approach would require a construction whose continuous certificate satisfies the stronger relaxation on both satisfiable and unsatisfiable source
instances, while discrete attainment continues to characterize satisfiability.
\bibliographystyle{plainurl}
\bibliography{references}
\end{document}